\documentclass[10pt]{article}
\usepackage{amsthm}
\usepackage{amssymb}
\usepackage{enumerate}
\usepackage{multirow}
\usepackage{bbm}
\usepackage{amsfonts}
\usepackage{balance}
\usepackage{footnote}
\usepackage{graphicx}
\usepackage{tablefootnote}
\usepackage{xr}
\usepackage{hyperref}
\usepackage{amsmath}

\newcommand{\ex}[1]{\mathbb{E}\left[ #1 \right] }
\newcommand{\norm}[1]{\left\lVert #1 \right\rVert}
\newcommand{\norms}[1]{\lVert #1 \rVert}
\newcommand{\Q}{ {\mathbf Q}}

\newcommand{\A}{ {\mathbf A}}
\newcommand{\s}{ {\mathbf S}}
\newcommand{\U}{ {\mathbf U}}
\newcommand{\UU}{ {\mathbf U}}
\newcommand{\ck}{ {\mathbf c}^{(k)} }
\newcommand{\inner}[2]{\langle #1, #2 \rangle}
\newcommand{\lep}[1]{\mathop  \le \limits^{(#1)}}

\newcommand{\ep}[1]{\mathop  = \limits^{(#1)}}

\newtheorem{theorem}{Theorem}
\newtheorem{claim}{Claim}
\newtheorem{lemma}{Lemma}
\newtheorem{definition}{Definition}
\newtheorem{proposition}{Proposition}
\newtheorem{remark}{Remark}
\newtheorem{assumption}{Assumption}
\begin{document}
%
\title{A Note on Stein's Method for Heavy-Traffic Analysis}
\author{Xingyu Zhou\\Department of ECE\\The Ohio State University\\zhou.2055@osu.edu 
\and Ness Shroff\\Department of ECE and CSE\\The Ohio State University\\
shroff.11@osu.edu}

\maketitle



\begin{abstract}
  In this note, we apply Stein's method to analyze the steady-state distribution of queueing systems in the traditional heavy-traffic regime. Compared to previous methods (e.g., drift method and transform method), Stein's method allows us to establish stronger results with simple and template proofs. In particular, we consider discrete-time systems in this note. We first introduce the key ideas of Stein's method for heavy-traffic analysis through a single-server system. Then, we apply the developed template to analyze both load balancing problems and scheduling problems. All these three examples demonstrate the power and flexibility of Stein's method in heavy-traffic analysis. In particular, we can see that one appealing property of Stein's method is that it combines the advantages of both the drift method and the transform method.
\end{abstract}

%


\section{Introduction}
Heavy-traffic analysis of queueing systems dates back to~\cite{kingman1962queues}, in which the author showed that the scaled waiting time of a $G$/$G$/$1$ system in heavy-traffic approaches an exponential random variable using diffusion approximations method. This method was then applied to a variety of works on parallel queues~\cite{foschini1978basic,chen2012asymptotic,reiman1984some,hanqin1989heavy,williams1998diffusion,harrison1998heavy,bramson1998state,stolyar2004maxweight}. The key idea behind diffusion approximations is to show that the scaled queue length process converges to a regulated Brownian motion. This process-level convergence often results in sample-path optimality in finite time. However, establishing the convergence of steady-state distribution requires the additional validation of the interchange of limits argument, which is often not taken (some exceptions include~\cite{gamarnik2006validity,budhiraja2009stationary}, in which the authors proved an interchange of limit argument for generalized Jackson networks with a fixed routing matrix).

Recently, the authors in~\cite{eryilmaz2012asymptotically} developed the so-called drift method, which purely relies on Lyapunov drift arguments and is able to directly obtain $\emph{steady-state}$ heavy-traffic results without considering the validation of the interchange of limits in diffusion approximations; see some applications in load balancing~\cite{zhou2017designing,zhou2018heavy,zhou2018flexible,maguluri2014heavy}, scheduling~\cite{maguluri2016heavy,wang2017heavy} and bandwidth sharing network~\cite{wang2018heavy}. In particular, the main idea behind the drift method is to set the mean drift of a particular test function to be zero in steady-state. Then, to obtain a tighter result in heavy-traffic, it often needs to establish the state-space collapse result, which roughly means that the system state would live on a lower-dimensional space rather than the original high-dimensional space. For example, by choosing the test function to be the quadratic function of the (weighted) sum queue lengths and establishing state-space collapses onto a one-dimensional subspace, drift method results in the first moment optimality of both Join-Shortest-Queue (JSQ) for the load balancing problem and MaxWeight for the scheduling problem~\cite{eryilmaz2012asymptotically}. That is, the first moment of (weighted) sum queue lengths in the parallel queueing systems converges in heavy-traffic to that of a hypothetical single-server system (which is a lower bound for the parallel queueing system).
Moreover, if a polynomial test function of degree $n+1$ is chosen, drift method results in the $n$th moment optimality of the control policy (e.g., JSQ and MaxWeight). Therefore, drift method enables us to inductively show convergence of the steady-state distribution.

Instead of considering all the polynomial test functions to inductively establish convergence of steady-state distribution, one could directly use the exponential test function (which naturally recovers all the polynomial functions via Taylor series). This is exactly the key idea behind the transform method introduced recently in~\cite{hurtado2018transform}. In particular,~\cite{hurtado2018transform} shows that under a certain condition for state-space collapse, the (weighted) sum queue lengths of the parallel queueing systems converges in distribution to an exponential random variable in heavy-traffic (which is in fact has the same distribution as the corresponding hypothetical single-server system). 

However, due to the use of exponential test function, transform method, compared to drift method, has to additionally establish an exponential type equation for the unused service and check the existence of moment generating function. Both of them need additional work, which makes transform method tend to lose the simplicity of drift method. Then, one may wonder if we can directly obtain the convergence of steady-state distribution (as in transform method) while maintaining the simplicity of drift method. Fortunately, the answer is yes and Stein's method comes to rescue. 
As it will become clear later, the main advantage of Stein's method is that it allows us to implicitly choose the exponential test function while actually working with a quadratic test function. In addition to the simplicity, Stein's method allows us to achieve even stronger results, that is, convergence in Wasserstein distance metric (which implies convergence in distribution) and the characterization of the convergence rate (which were not obtained in previous methods).

Stein's method is a theoretical tool to obtain bounds on the distance between two probability distributions~\cite{stein1972bound}. It was first introduced to study queueing system by~\cite{gurvich2014diffusion}. Steady-state diffusion approximations for Erlang-A and Erlang-C models were investigated by Stein's method in~\cite{braverman2017stein}. For load balancing problems, Stein's method has been used in different asymptotic regimes including large-system regime~\cite{ying2016approximation}, many-server heavy-traffic regime~\cite{liu2019universal,braverman2018steady}. In the traditional heavy-traffic regime, using Stein's method, a single-server system was studied in the continuous-time setting~\cite{gaunt2020stein}.

In this note, we consider discrete-time queueing systems in the traditional heavy-traffic regime with Stein's method. We first apply Stein's method to analyze the distance of the steady-state distribution of the single-server system to an exponential distribution. Note that this distance bound is universal for all the traffic loads, and, in the heavy-traffic limit, the bound goes to zero. Similar to~\cite{braverman2017stein}, key steps in applying Stein's method for heavy-traffic analysis are identified through this single-server example, which serve as a template for more general problems. Based on this template, we then apply Stein's method to obtain the convergence of steady-state distribution for both load balancing problems and scheduling problems in parallel queueing systems. Through these examples, we can observe that besides the additional gradient bounds involved in Stein's method (which are quite standard and easy to obtain), all the other bounds follow exactly from those obtained by choosing a quadratic test function as in~\cite{eryilmaz2012asymptotically}. Thus, in some sense, Stein's method allows us to establish steady-state distribution results by just working on a quadratic test function rather than directly dealing with all the polynomial functions as in drift method or an exponential test function as in transform method. This is achieved by utilizing the gradient bounds for the solution to the Stein equation. 

The rest of the note is organized as follows. In Section~\ref{sec:model}, general models and preliminaries behind all the three examples are introduced. In Section~\ref{sec:single}, Stein's method is applied to analyze the steady-state distribution of a single-server system. Moreover, a template is developed for general problem. In Sections~\ref{sec:lb} and~\ref{sec:sche}, this template is adopted to study steady-state distribution of load balancing and scheduling problems, respectively.

\section{General model and preliminaries}
\label{sec:model}
In this section, we summarize the common model features behind the single-server problem, load balancing problem and scheduling problem. Moreover, we introduce some necessary notations and preliminaries before we dive into each problem.

We consider a single-hop queueing system in the discrete time, i.e., a time-slotted system. There are $N \ge 1$ separate servers, each of them maintains an infinite capacity FIFO queue. Once a task or job is in a queue, it remains in that queue until its service is completed. Each server is
assumed to be work conserving, i.e., a server is idle if and only if its corresponding queue is empty.

Let $Q_n(t)$ be the queue length (i.e., tasks in the queue and the server) of server $n$ at the beginning of time-slot $t$. Let $A_n(t)$ denote the number of arrivals in time-slot $t$ and $S_n(t)$ denote the amount of service that server n offers for queue $n$ in time-slot $t$. That is, $S_n(t)$ is the maximum number of tasks that can be completed by server $n$ at time-slot $t$. For ease of exposition, we assume that $A_n(t)$ and $S_n(t)$ both have finite supports as the main purpose of this paper is to demonstrate the key idea of Stein's method applied to heavy-traffic analysis. In the general cases with light-tailed distributions, a weaker result of the convergence rate can also be established. 

In each time-slot, the order of events is as follows. First, queue lengths (or partial queue lengths) are observed. Based on these observations, a control problem is solved (i.e., the load balancing problem or the scheduling problem). Then, arrivals happen and the server processes tasks at the end of each time slot. In particular, the evolution of the length of queue $n$ is given by
\begin{align}
  Q_n(t+1) = Q_n(t) + A_n(t) - S_n(t) + U_n(t),
\end{align}
where $U_n(t) = \max(S_n(t)-A_n(t)-Q_n(t),0)$ is the unused service due to an empty queue.

In this paper, we add a line on top of variables and vectors to denote steady-state (e.g., $\overline{\Q}$, $\overline{\A}$ and $\overline{\s}$). In order to perform our heavy-traffic analysis, we consider a set of systems parametrized by a positive parameter $\epsilon$. In particular, the parameter $\epsilon$ captures the distance of arrival vector to a particular point on the capacity region, i.e., a smaller $\epsilon$ means a heavier load. 

\begin{definition}
  A control policy is said to be throughput optimal if for any $\epsilon>0$, the system is positive recurrent and all the moments of $\norms{\overline{\Q}^{(\epsilon)}}$ are finite. 
\end{definition}

The main convergence metric used in this paper is the Wasserstein distance metric, which is defined as follows for non-negative random variables.
\begin{align*}
    d_W(X,Y) = \sup_{h\in\text{Lip}(1)} |\ex{h(X)} - \ex{h(Y)}|
  \end{align*}
where for a metric space $(\mathcal{S},d)$, $\text{Lip}(1)=\{h: \mathcal{S}\to\mathbb{R}, |h(x)-h(y)|\le d(x,y)\}$. The class $\text{Lip}(1)$ is simple to work with but at the same time rich enough so that convergence under the Wasserstein metric implies the convergence in distribution~\cite{gibbs2002choosing}.

\section{Stein's method for a single-server system}
\label{sec:single}
In this section, we apply Stein's method to a single-server problem. The main purpose of this section is to provide the template when applying Stein's method for heavy-traffic analysis, which can be followed to conduct heavy-traffic analysis for load balancing and scheduling problems later as well.

To distinguish from the multiple-server case, we will use lower-case letter (e.g., $q(t)$, $a(t)$) in this section. In particular, $a(t)$ is an integer-valued random variable, which is $\emph{i.i.d}$ across time-slots with mean of $\lambda$ and $s(t)$ is also a sequence of $\emph{i.i.d}$ random variables with mean of $\mu$. The arrival and service process are independent of each other and the queue length. We assume that that $a(t) \le A_{max}$ and $s(t) \le S_{max}$ for all $t$.
We consider a set of single-server system $\{q^{(\epsilon)}(t), t\ge 0\}$ parameterized by $\epsilon$ such that $\lambda^{(\epsilon) } = \mu - \epsilon$. Let $\bar{q}^{(\epsilon)}$, $\bar{a}^{(\epsilon)}$ and $\bar{s}$ denote random variable whose distribution is the same as the steady-state distribution of $\{q^{(\epsilon)}(t), t\ge 0\}$, $\{a^{(\epsilon)}(t), t\ge 0\}$ and $\{s(t), t\ge 0\}$. In particular, we have $\lambda^{(\epsilon)} = \ex{\bar{a}^{(\epsilon)}}$, $(\sigma_a^{(\epsilon)})^2 = \text{Var}[\bar{a}^{(\epsilon)}]$, $\mu = \ex{\bar{s}}$ and $\sigma_s^2 = \text{Var}[\bar{s}]$.
We assume that as $(\sigma_a^{(\epsilon)})^2$ approaches $\sigma_a^2$ as $\epsilon$ goes to zero.


  \begin{theorem}
  \label{thm:single}
    Consider the single-server system as described above and $Z \sim \text{Exp}(\frac{2 }{(\sigma_a^{(\epsilon)})^2 + \sigma_s^2 })$. Then, there exists a constant $K$ such that
    \begin{align*}
       d_W(\epsilon \bar{q}^{(\epsilon)},Z) \le K \epsilon.
    \end{align*}
  \end{theorem}
  \begin{proof}
    Consider the following Stein equation with $f_h'(0) = 0$ (\textbf{Step 1:} Stein equation (or Poisson equation))
    \begin{align}
    \label{eq:stein}
      \frac{1}{2}\sigma^2 f_h''(x)-\theta f_h'(x) = h(x) - \ex{h(Z)}.
    \end{align}
    If we replace $x$ by the random variable $\epsilon \bar{q}$ and take expectation on both sides, we obtain 
    \begin{align*}
      \ex{h(\epsilon\bar{q})} - \ex{h(Z)} = \ex{\frac{1}{2}\sigma^2 f_h''(\epsilon\bar{q})-\theta f_h'(\epsilon\bar{q})}
    \end{align*}
    Now, we can focus on the RHS of the equation above. (\textbf{Step 2:} Generator coupling)
    \begin{align}
    \label{eq:decomp}
      &\ex{\frac{1}{2}\sigma^2 f_h''(\epsilon\bar{q})-\theta f_h'(\epsilon\bar{q})}\nonumber\\
      \ep{a}&\ex{\frac{1}{2}\sigma^2 f_h''(\epsilon\bar{q})-\theta f_h'(\epsilon\bar{q}) - \left(f_h(\epsilon\bar{q}(t+1)) - f_h(\epsilon\bar{q}(t))\right)}\nonumber\\
      \ep{b}&\ex{\frac{1}{2}\sigma^2 f_h''(\epsilon\bar{q})-\theta f_h'(\epsilon\bar{q})}- \ex{f_h(\epsilon(\bar{q}(t) + \bar{a}(t) -\bar{s}(t) +\bar{u}(t)) - f_h(\epsilon\bar{q}(t))}
    \end{align}
    where (a) holds since $\bar{q}$ is in steady-state with all the moments bounded and the result (a) in Lemma~\ref{lem:basis}; (b) directly follows from the queueing dynamics.

    For the second expectation, we have (\textbf{Step 3:} Taylor expansion)
    \begin{align*}
      &\ex{f_h\left(\epsilon(\bar{q}(t) + \bar{a}(t) -\bar{s}(t) +\bar{u}(t)\right) - f_h\left(\epsilon\bar{q}(t)\right)}\nonumber\\
      =&\ex{f_h\left(\epsilon(\bar{q}(t) + \bar{a}(t) -\bar{s}(t))\right)-f_h\left(\epsilon\bar{q}(t)\right)} - \ex{f_h\left(\epsilon(\bar{q}(t) + \bar{a}(t) -\bar{s}(t))\right)- f_h\left(\epsilon(\bar{q}(t) + \bar{a}(t) -\bar{s}(t) +\bar{u}(t)\right)}\nonumber\\
      =& \ex{\epsilon f_h'\left(\epsilon\bar{q})(\bar{a}-\bar{s}\right) + \epsilon^2\frac{f_h''(\epsilon\bar{q})}{2}\left(\bar{a}-\bar{s}\right)^2 + \epsilon^3\frac{f_h'''(\eta)}{6} \left(\bar{a}-\bar{s}\right)^3} - \ex{-\epsilon \bar{u}f_h'(\epsilon\bar{q}(t+1)) + \epsilon^2\frac{f_h''(\xi)}{2}\bar{u}^2}\\
      =&\ex{\epsilon f_h'\left(\epsilon\bar{q})(\bar{a}-\bar{s}\right) + \epsilon^2\frac{f_h''(\epsilon\bar{q})}{2}\left(\bar{a}-\bar{s}\right)^2} + \ex{\epsilon^3\frac{f_h'''(\eta)}{6} \left(\bar{a}-\bar{s}\right)^3 + \epsilon \bar{u}f_h'(\epsilon\bar{q}(t+1)) - \epsilon^2\frac{f_h''(\xi)}{2}\bar{u}^2}\\
      \ep{a}& \ex{\epsilon^2\frac{f_h''(\epsilon\bar{q})}{2}\left((\sigma_a^{(\epsilon)})^2+\sigma_s^2+\epsilon^2\right) - \epsilon^2 f_h'(\epsilon\bar{q})} + \ex{\epsilon^3\frac{f_h'''(\eta)}{6} \left(\bar{a}-\bar{s}\right)^3 + \epsilon \bar{u}f_h'(\epsilon\bar{q}(t+1)) - \epsilon^2\frac{f_h''(\xi)}{2}\bar{u}^2}\\
      =&\ex{\epsilon^2\frac{f_h''(\epsilon\bar{q})}{2}\left((\sigma_a^{(\epsilon)})^2+\sigma_s^2\right) - \epsilon^2 f_h'(\epsilon\bar{q})} + \ex{\epsilon^4\frac{f_h''(\epsilon\bar{q})}{2}+ \epsilon^3\frac{f_h'''(\eta)}{6} \left(\bar{a}-\bar{s}\right)^3 + \epsilon \bar{u}f_h'(\epsilon\bar{q}(t+1)) - \epsilon^2\frac{f_h''(\xi)}{2}\bar{u}^2}
    \end{align*}
    where (a) follows from the fact that arrival and service are independent of each other and the queue length, together with $\ex{\bar{a}-\bar{s}} = -\epsilon$ and $\ex{(\bar{a}-\bar{s})^2} = (\sigma_a^{(\epsilon)})^2+\sigma_s^2+\epsilon^2$.

    Thus, if we let $\sigma^2 = \epsilon^2\left((\sigma_a^{\epsilon})^2+\sigma_s^2\right)$ and $\theta = \epsilon^2$ in Eq.~\eqref{eq:decomp}, we have (\textbf{Step 4:} Approximation)
    \begin{align*}
      \left|\ex{h(\epsilon\bar{q})} - \ex{h(Z)}\right| &= \left|\ex{\epsilon^3\frac{f_h'''(\eta)}{6} \left(\bar{a}-\bar{s}\right)^3 + \epsilon \bar{u}f_h'(\epsilon\bar{q}(t+1)) - \epsilon^2\frac{f_h''(\xi)}{2}\bar{u}^2} \right|\\
      &\le \underbrace{\ex{\left|\epsilon^4\frac{f_h''(\epsilon\bar{q})}{2}\right| + \left|\epsilon^3\frac{f_h'''(\eta)}{6} \left(\bar{a}-\bar{s}\right)^3\right| + \left|\epsilon^2\frac{f_h''(\xi)}{2}\bar{u}^2\right|}}_{\mathcal{T}_1} + \underbrace{\ex{\left|\epsilon \bar{u}f_h'(\epsilon\bar{q}(t+1))\right|}}_{\mathcal{T}_2}
    \end{align*}

    If $h$ is Lipschitz, the solution $f_h$ satisfies the gradient bounds in Lemma~\ref{lem:basis}. In particular, we consider the function class $h\in \text{Lip}(1)$. (\textbf{Step 4:} Gradient bounds) 

    For $\mathcal{T}_1$, we have
    \begin{align*}
      \mathcal{T}_1 &\le \epsilon^4\frac{\norm{f_h''}}{2} +\epsilon^3\frac{\norm{f_h'''}}{6} \ex{\bar{a}^3+\bar{s}^3 + 3\mu\left(\bar{a}^2+\bar{s}^2\right)} + \epsilon^2\frac{\norm{f_h''}}{2}\ex{\bar{u}^2}\\
      &\lep{a} \frac{1}{2}\epsilon^2 + \frac{2\epsilon}{3\left( (\sigma_a^{(\epsilon)})^2+\sigma_s^2\right)}\ex{\bar{a}^3+\bar{s}^3 + 3\mu\left(\bar{a}^2+\bar{s}^2\right)} + \frac{1}{2}\ex{\bar{u}^2}\\
      &\lep{b} K_1\epsilon + S_{max}\ex{\bar{u}}\\
      &\lep{c} K\epsilon
    \end{align*}
    where (a) follows from the gradient bounds in Lemma~\ref{lem:basis}; in (b) the constant $K_1$ follows from the fact that all the moments of arrival and service are bounded, and the fact that $u(t)\le s(t) < S_{max}$; (c) holds since $\ex{\bar{u}} = \epsilon$, which is obtained by setting the mean drift of $\bar{q}$ to be zero in steady-state, and $K = \max(K_1, S_{max})$.

    For $\mathcal{T}_2$, we have 
    \begin{align*}
      \mathcal{T}_2 &\ep{a} \ex{\left|\epsilon \bar{u}f_h'(\epsilon\bar{q}(t+1)) - \epsilon\bar{u}f_h'(0)\right|}\\
      &\ep{b}\ex{\left|\epsilon \bar{u}(t)f_h''(\zeta)\epsilon \bar{q}(t+1)\right|}\\
      &\ep{c}0
    \end{align*}
    where (a) holds since $f_h'(0) = 0$; (b) follows from the mean-value theorem; (c) is the result of $u(t)q(t+1) = 0$ for all $t$, which holds by the definition of $u(t)$.

    Thus, we have 
    \begin{align*}
       |\ex{h(\epsilon\bar{q})} - \ex{h(Z)}| &\le \mathcal{T}_1 + \mathcal{T}_2 \le K\epsilon
    \end{align*}
    which completes the proof of Theorem~\ref{thm:single}.
    \end{proof}

    \begin{remark}
      Note that the distance bound holds for any $\epsilon>0$. Moreover, except the gradient bounds, all other bounds are exactly the same as in the case of a quadratic test function in~\cite{eryilmaz2012asymptotically}. This nice property still holds in the proofs for load balancing and scheduling problems.
    \end{remark}

  \begin{lemma}
  \label{lem:basis}
    Let $f_h$ be the solution of the Stein equation given by Eq.~\eqref{eq:stein}. If $h$ is Lipschitz, we have 
    \begin{enumerate}[(a)]
      \item $|f_h'(x)| \le \frac{\sigma^2+2\theta x}{2\theta^2}\norm{h'}$
      \item $\norm{f_h''} \le \frac{\norm{h'}}{\theta}$
      \item $\norm{f_h'''} \le \frac{4\norm{h'}}{\sigma^2}$
    \end{enumerate}
  \end{lemma}
  \begin{proof}
    See Appendix~\ref{proof:lem1}
  \end{proof}

    From the proof of Theorem~\ref{thm:single}, we can see that it follows essentially the same procedures as summarized in~\cite{braverman2017stein}. They include Stein equation, generator coupling, Taylor expansion, approximation and gradient bounds. In the following, we present more details and insights behind each of them when applied to heavy-traffic analysis in general. 

    \textbf{Stein equation.} This is the cornerstone in Stein's method. The key intuition is from the following $\emph{characterizing equation}$ for an exponential distribution with mean of $\frac{\sigma^2}{2\theta}$. Specifically, suppose a random variable $Z \sim \text{Exp}(\frac{2\theta}{\sigma^2})$, i.e., with mean of $\frac{\sigma^2}{2\theta}$, then 
    \begin{align}
    \label{eq:chac}
      \ex{\frac{1}{2}\sigma^2f''(Z) -\theta f'(Z) + \theta f'(0)} = 0
    \end{align}
    holds for all functions $f: \mathbb{R}^+ \to \mathbb{R}$ with Lipschitz derivative. In fact, $Z \sim \text{Exp}(\frac{2\theta}{\sigma^2})$ can be viewed as the stationary distribution of a reflected Brownian motion (RBM) with a drift of $\theta$ and variance of $\sigma^2$. The corresponding generator is given by $Gf(x) = \frac{1}{2}\sigma^2f''(x) -\theta f'(x)$ for $x\ge 0$ and $f'(0) = 0$.   Thus, the motivation behind the choice of the particular Stein equation (i.e., Eq.~\eqref{eq:stein}) is that if the random variable $x$ approximates $Z$, then the RHS should be approximately zero. Meanwhile, by Eq.~\eqref{eq:chac} the LHS is also zero when $x$ is an exponential random variable with mean of $\frac{\sigma^2}{2\theta}$. As a result of Stein equation, bounding the distance between $x$ and $Z$ measured by $h$ is now equivalent to bounding the error when applying the generator $Gf(.)$ to the random variable $x$.

    \textbf{Generator coupling.} In this step, we couple the generator of the RBM and the generator of the single-server system. This is possible because in steady-state, the mean drift measured by the function $f_h$ is zero.

    \textbf{Taylor expansion.} In this step, we do the Taylor expansion on the generator of the single-server system in the hope that the expansion yields similar structures with the generator of the RBM.

    \textbf{Approximation and gradient bounds.} In this step, we carefully choose the parameter for the generator of RBM so that it matches some of the terms obtained by Taylor expansion. The difference between the two generators is now captured by terms that involve with the gradients of the solution to the Stein equation, which can be easily bounded by using the property of the solution $f_h$.

    We end this section by providing a more general result compared to Theorem~\ref{thm:single}. As mentioned before, we assume that both the arrival and service has finite support in Theorem~\ref{thm:single}. However, we can relax this assumption to the following light-tail assumption. In words, it says that all the moments of the arrival and service processes are bounded by constants that are independent of $\epsilon$.
    \begin{assumption}[Light-tail assumption]
      The arrival process $a(t)$ and service process $s(t)$ satisfy that 
      \begin{align*}
        \ex{e^{\theta_1 a(t)}} \le D_1 \text{  and } \ex{e^{\theta_2 s(t)}} \le D_2,
      \end{align*}
      for some constants $\theta_1 > 0$, $\theta_2>0$, $D_1<\infty$ and $D_2<\infty$ that are all independent of $\epsilon$.
    \end{assumption}

    \begin{theorem}
     Consider a single-server system that satisfies the light-tail assumption. Let $Z \sim \text{Exp}(\frac{2 }{(\sigma_a^{(\epsilon)})^2 + \sigma_s^2 })$, then 
      \begin{align*}
       d_W(\epsilon \bar{q}^{(\epsilon)},Z) = O(\epsilon\log{\frac{1}{\epsilon}}).
    \end{align*}
    \end{theorem}
    \begin{proof}
      As in the proof of Theorem~\ref{thm:single}, we need only focus on the two terms $\mathcal{T}_1$ and $\mathcal{T}_2$. For $\mathcal{T}_2$, it is still zero as before. 

      For $\mathcal{T}_1$, since the light-tail assumption implies that all the moments are bounded, we have  
      \begin{align*}
        \mathcal{T}_1 \le O(\epsilon) + \frac{1}{2}\ex{\bar{u}^2}.
      \end{align*}
      Now, we need a careful analysis of the second term $\ex{\bar{u}^2}$ since $\bar{u}$ is no longer upper bounded by $S_{max}$. First, note that for any $t$ and a constant $s'$,
    \begin{align*}
      u^2(t) &\le u(t)s(t)\\
      & = u(t)s(t)\mathcal{I}(s(t) \le s') + u(t)s(t)\mathcal{I}(s(t) > s')\\
      &\le u(t)s' + s^2(t)\mathcal{I}(s(t) > s').
    \end{align*}
    Thus, 
    \begin{align*}
      \ex{\bar{u}^2} &\le \epsilon s' + \ex{\bar{s}^2\mathcal{I}(\bar{s} > s')}\\
      &\le \epsilon s' + \sqrt{\ex{\bar{s}^4}}\sqrt{\mathbb{P}(\bar{s} > s')}\\
      &\le \epsilon s' + \sqrt{\ex{\bar{s}^4}} \sqrt{\frac{D_2}{e^{\theta_2 s'}}}.
    \end{align*}
    Let $s' = \frac{\log(D_2/\epsilon^2)}{\theta_2}$, we have $\ex{\bar{u}^2} = O(\epsilon\log{\frac{1}{\epsilon}})$. Therefore, $d_W(\epsilon \bar{q}^{(\epsilon)},Z) \le O(\epsilon\log{\frac{1}{\epsilon}})$.
    \end{proof}
\begin{remark}
  The same trick can also be applied to analyze the convergence for load balancing and scheduling problem in heavy traffic with light-tailed distributions.
\end{remark}

\section{Stein's method for the load balancing problem}
\label{sec:lb}
In this section, we will apply the template of Stein's method developed in the last section to analyze load balancing problems in heavy traffic. We will see that the proof of the main result follows nearly the same pattern as in the single-server case. In particular, compared to the single-server system, there is only one additional error term, which can often be bounded by using the condition of $\emph{state space collapse}$.

Let $A_{\Sigma}(t)$ denote denote the number of exogenous tasks that arrive at the beginning of time-slot $t$.  We assume that $A_{\Sigma}(t)$ is an integer-valued random variable with mean of $\lambda_{\Sigma}$, which is i.i.d. across time-slots. We further assume that there is a positive probability for $A_{\Sigma}(t)$ to be zero. We assume that $S_n(t)$
is also an integer-valued random variable with mean $\mu_n$, which is i.i.d. across time-slots. We also assume that $S_n(t)$ is independent across different servers as well as the arrival process. Let $S_{\Sigma}(t) \triangleq \sum_{n=1}^N S_n(t)$ denote the hypothetical total service process with mean of $\mu_{\Sigma} \triangleq \sum_{n=1}^N \mu_n$.

We consider a set of load balancing systems parameterized by $\epsilon$ such that $\lambda_{\Sigma}^{(\epsilon)} = \mu_{\Sigma} - \epsilon$. In particular, we have $\lambda_{\Sigma}^{(\epsilon)} = \ex{\overline{A}_{\Sigma}}$, $(\sigma_{\Sigma}^{(\epsilon)})^2 = \text{Var}(\overline{A}_{\Sigma})$, $\mu_{\Sigma} = \ex{\overline{S}_{\Sigma}}$ and $\nu_{\Sigma}^2 = \text{Var}({\overline{S}_{\Sigma}})$. A load balancing policy is adopted by the dispatcher to determine to which queue the new arrivals should be sent.

\begin{theorem}
\label{thm:lb}
    Consider a set of load balancing systems parameterized by $\epsilon$. Suppose that the load balancing policy is throughput optimal and there exists a function $g(\epsilon)$ such that 
    \begin{align}
    \label{eq:cross}
      \ex{\norms{\overline{\Q}^{(\epsilon)}(t+1)}_1\norms{\overline{\U}^{(\epsilon)}}_1} = O(g(\epsilon)).
    \end{align}
    Then, we have 
    \begin{align*}
       d_W(\epsilon \sum_{n=1}^N\overline{Q}_n^{(\epsilon)},Z) = O(\max(g(\epsilon),\epsilon)).
    \end{align*}
    where $Z \sim \text{Exp}(\frac{ (\sigma_{\Sigma}^{(\epsilon)})^2 + \nu_{\Sigma}^2 }{2})$. 
    
  \end{theorem}
  \begin{proof}
    Replace $x$ in the Stein equation (i.e., Eq.~\eqref{eq:stein}) by $\epsilon\norms{\overline{\Q}^{(\epsilon)}}_1$ and take expectation of both sides, we have 
    \begin{align}
    \label{eq:stein_routing}
      \left|\ex{h(\epsilon\norms{\overline{\Q}^{(\epsilon)}}_1)} - \ex{h(Z)}\right| = \left|\ex{\frac{1}{2}\sigma^2 f_h''\left(\epsilon\norms{\overline{\Q}^{(\epsilon)}}_1\right)-\theta f_h'\left(\epsilon\norms{\overline{\Q}^{(\epsilon)}}_1\right)}\right|
    \end{align}
    Now, we focus on the RHS. In particular, we have
    \begin{align*}
      &\ex{\frac{1}{2}\sigma^2 f_h''\left(\epsilon\norms{\overline{\Q}^{(\epsilon)}}_1\right)-\theta f_h'\left(\epsilon\norms{\overline{\Q}^{(\epsilon)}}_1\right)}\\
      \ep{a}&\ex{\frac{1}{2}\sigma^2 f_h''\left(\epsilon\norms{\overline{\Q}^{(\epsilon)}}_1\right)-\theta f_h'\left(\epsilon\norms{\overline{\Q}^{(\epsilon)}}_1\right) - \left(f_h\left(\epsilon\norms{\overline{\Q}^{(\epsilon)}(t+1)}_1\right) -f_h\left(\epsilon\norms{\overline{\Q}^{(\epsilon)}(t)}_1\right) \right)}\\
      =&\ex{\frac{1}{2}\sigma^2 f_h''\left(\epsilon\norms{\overline{\Q}}_1\right)-\theta f_h'\left(\epsilon\norms{\overline{\Q}}_1\right)} - \ex{f_h\left(\epsilon(\norms{\overline{\Q}(t)}_1 + \norms{\overline{\A}(t)}_1 - \norms{\overline{\s}(t)}_1 + \norms{\overline{\U}(t)}_1)\right) - f_h\left(\epsilon\norms{\overline{\Q}}_1\right)}
    \end{align*}
    where (a) holds since the policy is throughput optimal and the result (a) in Lemma~\ref{lem:basis}.

    For the second expectation, we can follow exactly the same argument as in the single-server case (i.e., replace the scalar by the $1$-norm of corresponding vectors) and obtain that
    \begin{align*}
      &\ex{f_h\left(\epsilon(\norms{\overline{\Q}(t)}_1 + \norms{\overline{\A}(t)}_1 - \norms{\overline{\s}(t)}_1 + \norms{\overline{\U}(t)}_1)\right) - f_h\left(\epsilon\norms{\overline{\Q}}_1\right)}\\
      =&{\ex{\epsilon^2\frac{f_h''(\epsilon\norms{\overline{\Q}}_1)}{2}\left(\norms{\overline{\A}}_1-\norms{\overline{\s}}_1 \right)^2 + \epsilon f_h'(\epsilon \norms{\overline{\Q}}_1)\left(\norms{\overline{\A}}_1-\norms{\overline{\s}}_1\right)}} \\
      &+\ex{\epsilon^3\frac{f_h'''(\eta)}{6} \left(\norms{\overline{\A}}_1-\norms{\overline{\s}}_1\right)^3 + \epsilon\norms{\overline{\U}}_1f_h'(\epsilon\norms{\overline{\Q}(t+1)}_1) -\epsilon^2\frac{f_h''(\xi)}{2}\norms{\overline{\U}}_1^2 }\\
      =&\ex{\epsilon^2\frac{f_h''(\epsilon\norms{\overline{\Q}}_1)}{2}\left( (\sigma_{\Sigma}^{(\epsilon)})^2 + \nu_{\Sigma}^2+ \epsilon^2\right) -\epsilon^2f_h'(\epsilon\norms{\overline{\Q}}_1)}\\
      &+\ex{\epsilon^4\frac{f_h''(\epsilon\norms{\overline{\Q}}_1)}{2} + \epsilon^3\frac{f_h'''(\eta)}{6} \left(\norms{\overline{\A}}_1-\norms{\overline{\s}}_1\right)^3 + \epsilon\norms{\overline{\U}}_1f_h'(\epsilon\norms{\overline{\Q}(t+1)}_1) -\epsilon^2\frac{f_h''(\xi)}{2}\norms{\overline{\U}}_1^2 }
    \end{align*}
    Now, let $\sigma^2 =\epsilon^2\left(\sigma_{\Sigma}^2 + \nu_{\Sigma}^2\right)$ and $\theta = \epsilon^2$ in Eq.~\eqref{eq:stein_routing}, we have 
    \begin{align*}
      \left|\ex{h(\epsilon\norms{\overline{\Q}^{(\epsilon)}}_1)} - \ex{h(Z)}\right| &\le \underbrace{\ex{\left|\epsilon^3\frac{f_h'''(\eta)}{6} \left(\norms{\overline{\A}}_1-\norms{\overline{\s}}_1\right)^3\right| + \left|\epsilon^2\frac{f_h''(\xi)}{2}\norms{\overline{\U}}_1^2\right| + \left|\epsilon^4\frac{f_h''(\epsilon\norms{\overline{\Q}}_1)}{2}\right|  }}_{\mathcal{T}_1}\\
      &+ \underbrace{\ex{\left|\epsilon\norms{\overline{\U}}_1f_h'(\epsilon\norms{\overline{\Q}(t+1)}_1)\right|}}_{\mathcal{T}_2}\nonumber
    \end{align*}

    For $\mathcal{T}_1$, we have 
    \begin{align*}
      \mathcal{T}_1 &\le \epsilon^3\frac{\norm{f_h'''}}{6}\ex{\overline{A}_{\Sigma}^3 + \overline{S}_{\Sigma}^3 + 3\mu_{\Sigma}(\overline{A}_{\Sigma}^2 + \overline{S}_{\Sigma}^2)} + \epsilon^2\frac{\norm{f_h''}}{2}\ex{\norms{\overline{\U}}_1^2} + \epsilon^4 \frac{\norm{f_h''}}{2}\\
      &\le \frac{2\epsilon}{3\left(\sigma_{\Sigma}^2 + \nu_{\Sigma}^2\right)}\ex{\overline{A}_{\Sigma}^3 + \overline{S}_{\Sigma}^3 + 3\mu_{\Sigma}(\overline{A}_{\Sigma}^2 + \overline{S}_{\Sigma}^2)} + \frac{1}{2}\ex{\norms{\overline{\U}}_1^2} + \frac{1}{2}\epsilon^2\\
      &\le O(\epsilon) + NS_{max}\ex{\norms{\overline{\U}}_1}\\
      &= O(\epsilon)
    \end{align*}

    For $\mathcal{T}_2$, we have
    \begin{align*}
      \mathcal{T}_2 &= \ex{\left|\epsilon\norms{\overline{\U}}_1f_h'(\epsilon\norms{\overline{\Q}(t+1)}_1) - \epsilon\norms{\overline{\U}}_1f_h'(0)\right|}\\
      &=\ex{\left|\epsilon^2\norms{\overline{\Q}(t+1)}_1\norms{\overline{\U}}_1f_h''(\zeta)\right|}\\
      &\le \ex{\norms{\overline{\Q}(t+1)}_1\norms{\overline{\U}}_1}\\
      & = O(g(\epsilon))
    \end{align*}
    Thus, we have 
    \begin{align*}
       \left|\ex{h(\epsilon\norms{\overline{\Q}^{(\epsilon)}}_1)} - \ex{h(Z)}\right| \le \mathcal{T}_1 + \mathcal{T}_2 = O(\max(g(\epsilon), \epsilon )),
    \end{align*}
    which completes the proof of Theorem~\ref{thm:lb}.
  \end{proof}
  \begin{remark}
    Note that Theorem~\ref{thm:lb} identifies the key term in studying the steady-state distribution for load balancing problems (i.e., Eq.~\eqref{eq:cross}). This term is often analyzed with the help of state-space collapse result. Specifically, the state-space collapse result states that for any $\epsilon \in (0,\epsilon_0)$, the system state will concentrate around a subspace, either a one-dimensional subspace~\cite{eryilmaz2012asymptotically,zhou2017designing} or a multi-dimensional subspace~\cite{zhou2018flexible}. In this case, it is also possible to obtain a universal bound for any $\epsilon \in (0,\epsilon_0)$. Furthermore, Theorem~\ref{thm:lb} also establishes a collection between the convergence rate of steady-state distribution with a new metric on load balancing policy developed in~\cite{zhou2018degree}.
  \end{remark}

\section{Stein's method for the scheduling problem}
\label{sec:sche}
In this section, we apply the template of Stein's method developed in the single-server case to analyze the scheduling problem. As before, similar patterns occur in the proof even though there are some additional terms to be bounded, which follow the bounds obtained as in the drift method~\cite{eryilmaz2012asymptotically}.

The goal of the scheduling problem is to select an instantaneous service rate vector $\s(t)$ at each time-slot, subject to feasibility constraints. Generally, let $\mathbf{\mathcal{S}}$ denote the set of feasible service rate vectors, and in each time-slot a particular service vector $\s \in \mathbf{\mathcal{S}}$ is selected, which is assumed to be non-negative integer-valued and bounded with $S_n \le S_{max} < \infty$. We assume that the arrival processes to different queues are independent and the sequence of $A_n(t)$ is $\emph{i.i.d}$, non-negative integer valued and bounded random variables with $A_n(t) \le A_{max} <\infty$. For the service process of each server $n$, the mean is $\mu_n$ and variance is $\nu_n^2$. The arrival process for each queue has mean $\lambda_{n}$ and variance of $\sigma_n^2$. Let $\pmb{\lambda} = (\lambda_n)_n$ and $\pmb{\sigma^2} = (\sigma_n^2)_n$ denote the vector for the mean and variance of the arrival process, and $\pmb{\mu} = (\mu_n)_n$ and $\pmb{\nu^2} = (\nu_n^2)_n$ for the service process.

Given the set of feasible service rate vectors $\mathbf{\mathcal{S}}$, the capacity region $\mathcal{R}$ is the convex hull of $\mathbf{\mathcal{S}}$ given by 
\begin{align*}
  \mathcal{R} = \text{Convex Hull}(\mathbf{\mathcal{S}}).
\end{align*}
By the nonnegative nature and finiteness of the set $\mathbf{\mathcal{S}}$, the capacity region $\mathcal{R}$ becomes a polyhedron, which can be described by 
\begin{align*}
  \mathcal{R} = \{\mathbf{r} \ge 0: \inner{\ck}{\mathbf{r}} \le b^{(k)}, k=1,2,\ldots,K \},
\end{align*}
where $K$ is the finite total number of hyperplanes that determine the polyhedron. For each hyperplane $\mathcal{H}^{(k)}$, it is characterized by its normal vector $\ck \in \mathbb{R}^N$ and the inner product value $b^{(k)}$. The intersection of the $k$th hyperplane with the capacity region is called the $k$th face of $\mathcal{R}$, given by 
\begin{align*}
  \mathcal{F}^{(k)} \triangleq \{\mathbf{r} \in \mathcal{R}: \inner{\ck}{\mathbf{r}} = b^{(k)}\}.
\end{align*}
Throughout this section, we fix a particular $\mathcal{F}^{(k)}$ and a point $\pmb{\lambda}^{(k)} \in \text{Relint} (\mathbf{F}^{(k)})$, where $\text{Relint} (\mathbf{F}^{(k)})$ denotes the relative interior of the polyhedral set $\mathcal{F}^{(k)}$.

As before, we consider a set of systems parameterized by $\epsilon$ such that the arrival vector $\pmb{\lambda}^{(\epsilon)}$ satisfies
\begin{align}
\label{eq:epsilon_schedule}
  \pmb{\lambda}^{(\epsilon)} \triangleq \pmb{\lambda}^{(k)} - \epsilon\ck,
\end{align}
which means that $\pmb{\lambda}^{(\epsilon)} \in \text{Int}(\mathcal{R})$. In words, $\pmb{\lambda}^{(\epsilon)}$ is a stabilizable rate in the capacity region that is at a distance $\epsilon$ away from the $k$th face $\mathcal{F}^{(k)}$.

To demonstrate the key idea, we consider the well-known scheduling policy called MaxWeight defined as 
\begin{align*}
  \s(t) = \text{RAND}\left\{ \arg\max_{\s \in \mathcal{S}} \inner{\Q(t)}{\s} \right\}.
\end{align*}

It has been shown that MaxWeight is throughput optimal and enjoys state-space collapse in~\cite{eryilmaz2012asymptotically}. In the following, we restate the state-space collapse result in~\cite{eryilmaz2012asymptotically} with a more precise statement on the constants, which is useful for establishing our main result.
\begin{proposition}
\label{prop:ssc}
For the fixed face  $\mathcal{F}^{(k)}$ and the point $\pmb{\lambda}^{(k)} \in \text{Relint} (\mathcal{F}^{(k)})$, consider the MaxWeight scheduling policy with arrival vector defined in Eq.~\eqref{eq:epsilon_schedule}. Let $\Q_{\parallel}^{(\epsilon,k)}$ be the projection of queue length vector $\Q^{(\epsilon)}$ onto the vector $\ck$ and $\Q_{\perp}^{(\epsilon, k)} \triangleq \Q - \Q_{\parallel}^{(\epsilon,k)}$. Then, there exist finite constants $\{M_r^{(k)}\}_{r = 1,2,\ldots}$, independent of $\epsilon$, such that 
  \begin{align}
  \label{eq:ssc}
    \ex{ \norm{\overline{\Q}_{\perp}^{(\epsilon,k)}}^r } \le M_r^{(k)},
  \end{align}
  for all $\epsilon > 0$ and any $r = 1,2,\ldots$. Moreover, $M_r^{(k)}$ is upper bounded as follows
  \begin{align}
  \label{eq:ssc_bound}
    M_r^{(k)} \le V_r^{(k)}r^{r+\frac{1}{2}}e^{1-r},
  \end{align}
  where $V_r^{(k)}$ is a constant given by $V_r^{(k)} = \left(\frac{4L}{\delta^{(k)}} + \frac{8D^2 + 4D\delta^{(k)}}{\delta^{(k)}}\right)^r$, $L = N\max(A_{max},S_{max})^2$, $D = \sqrt{N}\max(A_{max},S_{max})$ and $\delta^{(k)}$ is a constant determined by the face  $\mathcal{F}^{(k)}$ and $\pmb{\lambda}^{(k)}$.
\end{proposition}
\begin{proof}
  See Appendix~\ref{proof:prop1}
\end{proof}

\begin{theorem}
  Consider a set of scheduling systems described above that are parametrized by $\epsilon$ defined in Eq.~\eqref{eq:epsilon_schedule}. Suppose the scheduling policy is MaxWeight and $Z\sim \text{Exp}( \frac{2}{\inner{(\ck)^2}{ (\pmb{\sigma}^{(\epsilon)})^2}}  )$, then 
  \begin{align*}
       d_W(\epsilon \inner{\ck}{\overline{\Q}^{(\epsilon)}},Z) = O\left(\epsilon\log{\frac{1}{\epsilon}}\right).
    \end{align*}
\end{theorem}
\begin{proof}
  Replace $x$ in the Stein equation (i.e., Eq.~\eqref{eq:stein}) by $\epsilon\inner{\ck}{{\overline{\Q}^{(\epsilon)}}}$ and take expectation of both sides, we have 
    \begin{align}
    \label{eq:stein_scheduling}
      \left|\ex{h(\epsilon\inner{\ck}{{\overline{\Q}^{(\epsilon)}}})} - \ex{h(Z)}\right| = \left|\ex{\frac{1}{2}\sigma^2 f_h''\left(\epsilon\inner{\ck}{{\overline{\Q}^{(\epsilon)}}}\right)-\theta f_h'\left(\epsilon\inner{\ck}{{\overline{\Q}^{(\epsilon)}}}\right)}\right|
    \end{align}
    Now, we focus on the RHS. In particular, we have
    \begin{align}
      &\ex{\frac{1}{2}\sigma^2 f_h''\left(\epsilon\inner{\ck}{{\overline{\Q}^{(\epsilon)}}}\right)-\theta f_h'\left(\epsilon\inner{\ck}{{\overline{\Q}^{(\epsilon)}}}\right)}\nonumber\\
      =&\ex{\frac{1}{2}\sigma^2 f_h''\left(\epsilon\inner{\ck}{{\overline{\Q}^{(\epsilon)}}}\right)-\theta f_h'\left(\epsilon\inner{\ck}{{\overline{\Q}^{(\epsilon)}}}\right)} - \ex{\left(f_h\left(\epsilon\inner{\ck}{{\overline{\Q}^{(\epsilon)}(t+1)}}\right) -f_h\left(\epsilon\inner{\ck}{{\overline{\Q}^{(\epsilon)}(t)}}\right) \right)}\label{eq:span},
    \end{align}
    which follows from the fact that the system is in steady state and MaxWeight policy is throughput optimal together with Lemma~\ref{lem:basis}.
    For the second expectation in Eq.~\eqref{eq:span}, we have
    \begin{align}
      &\ex{f_h\left(\epsilon\inner{\ck}{{\overline{\Q}^{(\epsilon)}(t+1)}}\right) -f_h\left(\epsilon\inner{\ck}{{\overline{\Q}^{(\epsilon)}(t)}}\right)}\nonumber\\
      =&\ex{f_h\left(\epsilon\inner{\ck}{{\overline{\Q}(t)} + \overline{\A}(t)- \overline{\s}(t)} \right) - f_h\left(\epsilon\inner{\ck}{{\overline{\Q}^{(\epsilon)}(t)}}\right)}\label{eq:first}\\
      &- \ex{f_h\left(\epsilon\inner{\ck}{{\overline{\Q}(t)} + \overline{\A}(t)- \overline{\s}(t)} \right)- f_h\left(\epsilon\inner{\ck}{{\overline{\Q}^{(\epsilon)}(t+1)}}\right) }\label{eq:second}
    \end{align}
    For Eq.~\eqref{eq:first}, by Taylor expansion, we have 
    \begin{align}
      &\ex{f_h\left(\epsilon\inner{\ck}{{\overline{\Q}(t)} + \overline{\A}(t)- \overline{\s}(t)} \right) - f_h\left(\epsilon\inner{\ck}{{\overline{\Q}^{(\epsilon)}(t)}}\right)}\nonumber\\
      =&\ex{\epsilon f_h'\left(\epsilon\inner{\ck}{{\overline{\Q}^{(\epsilon)}(t)}}\right) \inner{\ck}{\overline{\A}(t)- \overline{\s}(t)}}\label{eq:first_cross}\\
       &+ \ex{\epsilon^2\frac{f_h''\left(\epsilon\inner{\ck}{{\overline{\Q}^{(\epsilon)}(t)}}\right)}{2}\left( \inner{\ck}{\overline{\A}(t)- \overline{\s}(t)}\right)^2}\label{eq:second_cross}\\
      &+\ex{\epsilon^3\frac{f_h'''\left(\eta\right)}{6}\left( \inner{\ck}{\overline{\A}(t)- \overline{\s}(t)}\right)^3}\nonumber.
    \end{align}
    Now, let us first focus on Eq.~\eqref{eq:first_cross}. In particular, we have 
    \begin{align}
      &\ex{\epsilon f_h'\left(\epsilon\inner{\ck}{{\overline{\Q}^{(\epsilon)}(t)}}\right) \inner{\ck}{\overline{\A}(t)- \overline{\s}(t)}}\nonumber\\
      =&\ex{\epsilon f_h'\left(\epsilon\inner{\ck}{{\overline{\Q}}}\right)\left(\inner{\ck}{\overline{\A}}-b^{(k)}\right)} + \ex{\epsilon f_h'\left(\epsilon\inner{\ck}{{\overline{\Q}}}\right)\left(b^{(k)}-\inner{\ck}{\overline{\s} }\right)}\nonumber\\
      \ep{a}&\ex{-\epsilon^2f_h'\left(\epsilon\inner{\ck}{{\overline{\Q}}}\right) } + \ex{\epsilon f_h'\left(\epsilon\inner{\ck}{{\overline{\Q}}}\right)\left(b^{(k)}-\inner{\ck}{\overline{\s} }\right)}\label{eq:first_cross_res},
    \end{align}
    where (a) follows from Eq.~\eqref{eq:epsilon_schedule}.
    For Eq.~\eqref{eq:second_cross}, we have
    \begin{align}
       &\ex{\epsilon^2\frac{f_h''\left(\epsilon\inner{\ck}{{\overline{\Q}^{(\epsilon)}(t)}}\right)}{2}\left( \inner{\ck}{\overline{\A}(t)- \overline{\s}(t)}\right)^2}\nonumber\\
       =&\ex{\epsilon^2\frac{f_h''\left(\epsilon\inner{\ck}{\overline{\Q}}\right)} {2}\left(  \left(\inner{\ck}{\overline{\A}}-b^{(k)}\right)^2 + \left(b^{(k)}-\inner{\ck}{\overline{\s}}\right)^2 + 2\left(\inner{\ck}{\overline{\A}}-b^{(k)}\right)\left(b^{(k)}-\inner{\ck}{\overline{\s}}\right) \right)}\nonumber\\
       =&\ex{\epsilon^2\frac{f_h''\left(\epsilon\inner{\ck}{\overline{\Q}}\right)} {2}\left(\inner{\ck}{\overline{\A}-\pmb{\lambda}} + \inner{\ck}{\pmb{\lambda}}-b^{(k)} \right)^2} + \ex{\epsilon^2\frac{f_h''\left(\epsilon\inner{\ck}{\overline{\Q}}\right)} {2}\left(b^{(k)}-\inner{\ck}{\overline{\s}}\right)^2}\nonumber\\
       &-2\epsilon^3\ex{\frac{f_h''\left(\epsilon\inner{\ck}{\overline{\Q}}\right)} {2}\left(b^{(k)}-\inner{\ck}{\overline{\s}}\right)}\nonumber\\
       =&\ex{\epsilon^2\frac{f_h''\left(\epsilon\inner{\ck}{\overline{\Q}}\right)} {2}\left(\inner{(\ck)^2}{  (\pmb{\sigma}^{(\epsilon)})^2} + \epsilon^2\right)} + \ex{\epsilon^2\frac{f_h''\left(\epsilon\inner{\ck}{\overline{\Q}}\right)} {2}\left(b^{(k)}-\inner{\ck}{\overline{\s}}\right)^2}\nonumber\\
        &-\epsilon^3\ex{{f_h''\left(\epsilon\inner{\ck}{\overline{\Q}}\right)}\left(b^{(k)}-\inner{\ck}{\overline{\s}}\right)}\label{eq:second_cross_res}.
    \end{align} 
    Now, combining Eqs.~\eqref{eq:first_cross}~\eqref{eq:first_cross_res}~\eqref{eq:second_cross} and~\eqref{eq:second_cross_res}, obtains Eq.~\eqref{eq:first} as 
    \begin{align}
      &\ex{f_h\left(\epsilon\inner{\ck}{{\overline{\Q}(t)} + \overline{\A}(t)- \overline{\s}(t)} \right) - f_h\left(\epsilon\inner{\ck}{{\overline{\Q}^{(\epsilon)}(t)}}\right)}\nonumber\\
      =&\ex{\epsilon^2\frac{f_h''\left(\epsilon\inner{\ck}{\overline{\Q}}\right)} {2}\left(\inner{(\ck)^2}{ (\pmb{\sigma}^{(k)})^2} + \epsilon^2\right)-\epsilon^2f_h'\left(\epsilon\inner{\ck}{{\overline{\Q}}}\right)}+\ex{\epsilon f_h'\left(\epsilon\inner{\ck}{{\overline{\Q}}}\right)\left(b^{(k)}-\inner{\ck}{\overline{\s} }\right)}\nonumber\\
      &+\ex{\epsilon^2\frac{f_h''\left(\epsilon\inner{\ck}{\overline{\Q}}\right)} {2}\left(b^{(k)}-\inner{\ck}{\overline{\s}}\right)^2} -\epsilon^3\ex{{f_h''\left(\epsilon\inner{\ck}{\overline{\Q}}\right)}\left(b^{(k)}-\inner{\ck}{\overline{\s}}\right)}\nonumber\\
      &+\ex{\epsilon^3\frac{f_h'''\left(\eta\right)}{6}\left( \inner{\ck}{\overline{\A}- \overline{\s}}\right)^3}\label{eq:first_res}.
    \end{align}

    We now turn to Eq.~\eqref{eq:second}. In particular, we have 

    \begin{align}
      &\ex{f_h\left(\epsilon\inner{\ck}{{\overline{\Q}(t)} + \overline{\A}(t)- \overline{\s}(t)} \right)- f_h\left(\epsilon\inner{\ck}{{\overline{\Q}^{(\epsilon)}(t+1)}}\right)}\nonumber\\
      \ep{a}&\ex{-\epsilon\inner{\ck}{\overline{\UU}(t)}f_h'\left(\epsilon\inner{\ck}{{\overline{\Q}^{(\epsilon)}(t+1)}}\right) + \epsilon^2\frac{f_h''(\xi)}{2}\left(\inner{\ck}{\overline{\UU}(t)} \right)^2 }\nonumber\\
      \ep{b}&\ex{-\epsilon^2\inner{\ck}{\overline{\UU}(t)}\inner{\ck}{{\overline{\Q}^{(\epsilon)}(t+1)}}f_h''(\zeta)} + \ex{\epsilon^2\frac{f_h''(\xi)}{2}\left(\inner{\ck}{\overline{\UU}(t)} \right)^2}\label{eq:second_res},
    \end{align}
    where (a) follows from Taylor expansion; (b) holds by the mean-value theorem and the fact $f_h^{\prime}(0) = 0$.

    Now, let $\sigma^2 = \epsilon^2(\inner{(\ck)^2}{ (\pmb{\sigma}^{(k)})^2})$, $\theta = \epsilon^2$ in Eq.~\eqref{eq:span} and combine Eqs.~\eqref{eq:first},~\eqref{eq:second},~\eqref{eq:first_res} and~\eqref{eq:second_res}, we have 
    \begin{align}
      &\left|\ex{h(\epsilon\inner{\ck}{{\overline{\Q}^{(\epsilon)}}})} - \ex{h(Z)}\right|\nonumber\\
      \le& \underbrace{\ex{\left|\epsilon f_h'\left(\epsilon\inner{\ck}{{\overline{\Q}}}\right)\left(b^{(k)}-\inner{\ck}{\overline{\s} }\right)\right|}}_{\mathcal{T}_1}\\
       &+ \underbrace{\ex{\left|\epsilon^2\inner{\ck}{\overline{\UU}(t)}\inner{\ck}{{\overline{\Q}^{(\epsilon)}(t+1)}}f_h''(\zeta)\right|}}_{\mathcal{T}_2} + \underbrace{\ex{\left|\epsilon^2\frac{f_h''(\xi)}{2}\left(\inner{\ck}{\overline{\UU}(t)} \right)^2 \right|}}_{\mathcal{T}_3}\nonumber\\
      &+ \underbrace{\ex{\left|\epsilon^2\frac{f_h''\left(\epsilon\inner{\ck}{\overline{\Q}}\right)} {2}\left(b^{(k)}-\inner{\ck}{\overline{\s}}\right)^2\right| } +\ex{\left|\epsilon^3{f_h''\left(\epsilon\inner{\ck}{\overline{\Q}}\right)} \left(b^{(k)}-\inner{\ck}{\overline{\s}}\right)\right| }}_{\mathcal{T}_4}\nonumber\\
      &+\underbrace{\ex{\left|\epsilon^3\frac{f_h'''\left(\eta\right)}{6}\left( \inner{\ck}{\overline{\A}- \overline{\s}}\right)^3\right| + \left|\epsilon^4\frac{f_h''\left(\epsilon\inner{\ck}{\overline{\Q}}\right)} {2} \right| } }_{\mathcal{T}_5}.
    \end{align}
    We are left with the task of bounding each of the terms. Note that, for each face $\mathcal{F}^{(k)}$ of the capacity region $\mathcal{R}$, there exists an angle $\theta^{(k)} \in (0,\pi/2]$ such that
    \begin{align}
    \label{eq:angle}
      \inner{\ck}{\s} = b^{(k)}, \quad \text{for all } \Q \text{ satisfying } \frac{\norms{\Q_{\parallel}^{(k)}}}{\norms{\Q}} \ge \cos(\theta^{(k)}).
    \end{align}
    For the simplicity of notation, let $\theta_{\mathbf{x},\mathbf{y}} \triangleq \arg\cos\left(\frac{\inner{\mathbf{x}}{\mathbf{y}}}{\norms{\mathbf{x}}\norms{\mathbf{y}} } \right)$.
    Now, for $\mathcal{T}_1$, we have,
    \begin{align}
      \mathcal{T}_1 &\le \epsilon^2\norm{f_h''}\ex{\left(\inner{\ck}{\overline{\Q}}\right)\left(b^{(k)}-\inner{\ck}{\overline{\s} }\right)}\nonumber\\
      &\ep{a}\epsilon^2\norm{f_h''}\ex{\norms{\overline{\Q}_{\perp}^{(k)}}\cot(\theta_{\overline{\Q},\overline{\Q}_{\parallel}^{(k)} } )\mathcal{I}(\theta_{\overline{\Q},\overline{\Q}_{\parallel}^{(k)} } > \theta^{(k)}) \left(b^{(k)}-\inner{\ck}{\overline{\s} }\right)}\nonumber\\
      &\le\epsilon^2\norm{f_h''}\cot(\theta^{(k)})\ex{\norms{\overline{\Q}_{\perp}^{(k)}}\left(b^{(k)}-\inner{\ck}{\overline{\s} }\right)}\nonumber\\
      &\lep{b} \epsilon^2\norm{f_h''}\cot(\theta^{(k)})\left(\ex{\left(b^{(k)}-\inner{\ck}{\overline{\s} }\right)^{r'} }\right)^{\frac{1}{r'}}\left(\ex{\norms{\overline{\Q}_{\perp}^{(k)}}_r^r}\right)^{\frac{1}{r}}\nonumber\\
      &\lep{c} \epsilon^2\norm{f_h''}\cot(\theta^{(k)})\left(\ex{\left(b^{(k)}-\inner{\ck}{\overline{\s} }\right)^{r'} }\right)^{\frac{1}{r'}} (M_r^{(k)})^{1/r}\nonumber\\
      &\lep{d} \cot(\theta^{(k)}) (M_r^{(k)})^{1/r}\left(\ex{\left(b^{(k)}-\inner{\ck}{\overline{\s} }\right)^{r'} }\right)^{\frac{1}{r'}}\label{eq:t1_schedule},
    \end{align}
    where (a) follows from the definitions of $\Q_{\perp}$ and the angle $\theta^{(k)}$ in Eq.~\eqref{eq:angle}; (b) holds by H\"older inequality for random vectors, and $r, r'\in(1,\infty)$ satisfy $1/r + 1/r' = 1$; (c) follows from the state-space collapse result in Eq.~\eqref{eq:ssc} and the fact that $\norms{\mathbf{x}}_{r_2} \le \norms{\mathbf{x}}_{r_1}$ if $r_1 < r_2$; (d) follows from the gradient bound in Lemma~\ref{lem:basis} with $\theta = \epsilon^2$.

    To bound Eq.~\eqref{eq:t1_schedule}, we use the following claim, the proof of which is given in Appendix~\ref{proof:clm1}.
    \begin{claim}
    \label{clm:c1}
      For any $r' > 1$, $\ex{\left(b^{(k)}-\inner{\ck}{\overline{\s} }\right)^{r'}} \le \beta_2 \epsilon$ for some constant $\beta_2$.
    \end{claim}
    Thus, combining Eq.~\eqref{eq:t1_schedule} and Claim~\ref{clm:c1}, yields
    \begin{align}
    \label{eq:t1_schedule_res}
    \mathcal{T}_1 &\le \cot(\theta^{(k)}) (M_r^{(k)})^{1/r} \beta_2^{\prime}\epsilon^{\frac{1}{r'}}\nonumber\\
    &\lep{a}\cot(\theta^{(k)})\beta_2^{\prime}\beta_1\epsilon^{\frac{1}{r'}}r^{1+\frac{1}{2r}}e^{\frac{1}{r}-1}\nonumber\\
    &\lep{b}2\cot(\theta^{(k)})\beta_2^{\prime}\beta_1\epsilon\log{\frac{1}{\epsilon}} \quad \forall \epsilon \le \epsilon_0\nonumber\\
    &=O\left(\epsilon\log{\frac{1}{\epsilon}}\right),
    \end{align}
    where (a) holds by the result in Eq.~\eqref{eq:ssc_bound} with $\beta_1 = (V_r^{(k)})^{1/r}$; (b) follows similar arguments in~\cite{hurtado2020logarithmic}. Specifically, we pick $r = \log{\frac{1}{\epsilon}}$ in (a), then RHS of (a) becomes $\cot(\theta^{(k)})\beta_2^{\prime}\beta_1\epsilon \log{\frac{1}{\epsilon}} h(\epsilon)$, in which $h(\epsilon) \triangleq \epsilon^{-\frac{1}{\log{\frac{1}{\epsilon}}}}e^{\frac{1}{\log{\frac{1}{\epsilon}}}-1}\left(\log{\frac{1}{\epsilon}}\right)^{\frac{1}{2\log{\frac{1}{\epsilon}}}}$. Now, since $\lim_{\epsilon \downarrow 0} h(\epsilon) =  e \times \frac{1}{e} \times 1 = 1$, and hence there exists an $\epsilon_0$ such that for all $\epsilon \le \epsilon_0$, $h(\epsilon) \le 2$.

    For $\mathcal{T}_2$, we have 
    \begin{align}
      \mathcal{T}_2 &\le \epsilon^2\norm{f_h''}\ex{\inner{\ck}{\overline{\UU}(t)}\inner{\ck}{{\overline{\Q}^{(\epsilon)}(t+1)}}}.
    \end{align}
    It can be further upper bounded by using the following result, the proof of which is given in Appendix~\ref{proof:clm2}.
    \begin{claim}
    \label{clm:c2}
      $\ex{\inner{\ck}{\overline{\UU}(t)}} \le \epsilon$ and $\ex{\inner{\ck}{\overline{\UU}(t)}\inner{\ck}{{\overline{\Q}^{(\epsilon)}(t+1)}}} = O(\epsilon\log{\frac{1}{\epsilon}})$
    \end{claim}
    Thus, combining the gradient bound in Lemma~\ref{lem:basis} and Claim~\ref{clm:c2}, yields
    \begin{align}
      \mathcal{T}_2 = O\left(\epsilon\log{\frac{1}{\epsilon}}\right).
    \end{align}
    For $\mathcal{T}_3$, we have 
    \begin{align}
      \mathcal{T}_3 &\le \frac{1}{2}\epsilon^2\norm{f_h''}\ex{\left(\inner{\ck}{\overline{\UU}(t)} \right)^2 }\nonumber\\
      &\le\frac{1}{2}\ex{\left(\inner{\ck}{\overline{\UU}(t)} \right)^2 }\nonumber\\
      &\le \frac{1}{2}\inner{\ck}{S_{max}\mathbf{1}}\ex{\inner{\ck}{\overline{\UU}}}\nonumber\\
      &\ep{a}O(\epsilon),
    \end{align}
    where (a) follows from the first result in Claim~\ref{clm:c2}.
    For $\mathcal{T}_4$, we have 
    \begin{align}
      \mathcal{T}_4 &\le \frac{1}{2}\epsilon^2\norm{f_h''}\ex{\left( b^{(k)}-\inner{\ck}{\overline{\s} }\right)^{2}}+\epsilon^3\norm{f_h''}\ex{\left| b^{(k)}-\inner{\ck}{\overline{\s} }\right|}\nonumber\\
      &\le \frac{1}{2}\ex{\left( b^{(k)}-\inner{\ck}{\overline{\s} }\right)^{2}} + \epsilon \ex{\left| b^{(k)}-\inner{\ck}{\overline{\s} }\right|}\nonumber\\
      &\ep{a}O(\epsilon),
    \end{align}
    where (a) follows from Claim~\ref{clm:c1} and the boundness of the service process.

    For $\mathcal{T}_5$, we have 
    \begin{align}
      \mathcal{T}_5 &\le \frac{1}{6}\epsilon^3\norm{f_h'''}\ex{\left( \inner{\ck}{\overline{\A}- \overline{\s}}\right)^3} + \frac{1}{2}\epsilon^4 \norm{f_h''}\nonumber\\
      &\le \frac{\epsilon}{\inner{(\ck)^2}{\pmb{\sigma}^2}}\ex{\left( \inner{\ck}{\overline{\A}- \overline{\s}}\right)^3} + \epsilon^2\nonumber\\
      &\ep{a}O(\epsilon),
    \end{align}
    where (a) follows from the boundness of arrival and service processes.
    Therefore, we finally have
    \begin{align*}
      \left|\ex{h(\epsilon\inner{\ck}{{\overline{\Q}^{(\epsilon)}}})} - \ex{h(Z)}\right| \le \mathcal{T}_1 + \mathcal{T}_2 +\mathcal{T}_3 +\mathcal{T}_4 +\mathcal{T}_5 = O\left(\epsilon\log{\frac{1}{\epsilon}}\right)
    \end{align*}
    which establishes the result.
\end{proof}

\section{Conclusion}
In this note, we have successfully applied Stein's method to analyze the steady-state distribution of the single-sever system, load balancing and scheduling in parallel server systems, respectively. All the proofs share the same template, which can be also adopted to analyze more general queueing systems.

\bibliographystyle{plain}
\bibliography{ref} 

\section*{Appendix}
\appendix
\section{Proof of Lemma~\ref{lem:basis}}
\label{proof:lem1}
\begin{proof}
  Let $\hat{h}(t) = h(t) - \ex{h(Z)}$. The unique solution to Stein equation is then given by 
  \begin{align*}
     f_h'(x) = -e^{\frac{2\theta}{\sigma^2}x}\int_{x}^{\infty}\frac{2}{\sigma^2}\hat{h}(t)e^{-\frac{2\theta}{\sigma^2}t}dt.
   \end{align*} 
   Without loss of generality, we can assume that $h(0) =0$, and hence $h(t) \le t \norm{h'}$. Thus, 
   \begin{align*}
     |f_h'(x)|&\le e^{\frac{2\theta}{\sigma^2}x}\int_{x}^{\infty}\frac{2}{\sigma^2}t\norm{h'}e^{-\frac{2\theta}{\sigma^2}t}dt\\
    &\le \frac{\sigma^2+2\theta x}{2\theta^2}\norm{h'}
    \end{align*}

   Take the derivative on both side of Eq.~\eqref{eq:stein}, we have
   \begin{align}
   \label{eq:deriv_stein}
     \frac{1}{2}\sigma^2 f_h'''(x)-\theta f_h''(x) = h'(x).
   \end{align}
   Hence, we have 
   \begin{align*}
     f_h''(x) = -e^{\frac{2\theta}{\sigma^2}x}\int_{x}^{\infty}\frac{2}{\sigma^2}{h'}(t)e^{-\frac{2\theta}{\sigma^2}t}dt
   \end{align*}
   which implies that 
   \begin{align}
   \label{eq:secd_bound}
     |f_h''(x)| &\le e^{\frac{2\theta}{\sigma^2}x}\int_{x}^{\infty}\frac{2}{\sigma^2}|{h'}(t)|e^{-\frac{2\theta}{\sigma^2}t}dt\nonumber\\
     &\le e^{\frac{2\theta}{\sigma^2}x}\norm{h'}\frac{1}{\theta}\int_{x}^{\infty}\frac{2\theta}{\sigma^2}e^{-\frac{2\theta}{\sigma^2}t}dt\nonumber\\
     &\le \frac{\norm{h'}}{\theta}.
   \end{align}
   Re-arranging Eq.~\eqref{eq:deriv_stein} and combining with Eq~\eqref{eq:secd_bound}, yields
   \begin{align*}
     \frac{1}{2}\sigma^2|f_h'''(x)| \le 2\norm{h'},
   \end{align*}
   which directly implies the result.
  \end{proof}

 \section{Proof of Proposition~\ref{prop:ssc}}  
\label{proof:prop1}

In the proof, we will use the following result, which can be easily obtained by Lemmas~2 and~3 in~\cite{maguluri2016heavy}. For completeness, we restate the result first.
\begin{lemma}
        \label{lem:basis_flexible}
          For an irreducible aperiodic and positive recurrent Markov chain $\{X(t), t \ge 0\}$ over a countable state space $\mathcal{X}$, which converges in distribution to $\overline{X}$,  and suppose $V: \mathcal{X} \rightarrow \mathbb{R}_{+}$ is a Lyapunov function. We define the $T$ time-slot drift of $V$ at $X$ as 
          \[\Delta V(X)\triangleq [V(X(t_0+T)) - V(X(t_0))] \mathcal{I}(X(t_0) = X),\]
          where $\mathcal{I}(.)$ is the indicator function. Suppose for some positive finite integer $T$, the $T$ time-slot drift of $V$ satisfies the following conditions:

          \begin{itemize}
            \item (C1) There exists an $\eta> 0$ and a $\kappa <  \infty$ such that for any $t_0 = 1,2,\ldots$ and for all $X \in \mathcal{X}$ with $V(X)\ge \kappa$, 
            \[\mathbb{E}\left[\Delta V(X) \mid X(t_0) = X\right]\le -\eta.\]
            \item (C2) There exists a constant $D < \infty$ such that for all $X\in \mathcal{X}$,
            \[\mathbb{P}(|\Delta V(X)| \le D) = 1.\]
          \end{itemize}

          Then $\{V(X(t)), t\ge0\}$ converges in distribution to a random variable $\overline{V}$, and all moments of $\overline{V}$ exist and are finite. More specifically, we have for any $r = 1,2,\ldots$
          \begin{equation}
          \label{eq:upper_siva}
            \ex{V(\overline{X})^r} \le (2\kappa)^r + (4D)^r\left(\frac{D+\eta}{\eta} \right)^r r!.
          \end{equation}
\end{lemma}

\begin{proof}
  The result in Eq.~\eqref{eq:ssc} has already been established in Proposition~2 of ~\cite{eryilmaz2012asymptotically}. In order to obtain the result in~\eqref{eq:ssc_bound}, we will use Lemma~\ref{lem:basis_flexible}. Let $T=1$, $X = \Q$ and $V(\Q) = \norms{\Q_{\perp}^{(k)}}$. In~\cite{eryilmaz2012asymptotically}, it has been shown that 
  \begin{align*}
    \ex{\Delta V(\Q) \mid \Q} &\le -\delta^{(k)} + \frac{L}{\norms{\Q_{\perp}^{(k)}}}\\
    &\le -\frac{\delta^{(k)}}{2}, \quad \text{for all }  \norms{\Q_{\perp}^{(k)}} \ge \frac{2L}{\delta^{(k)}}
  \end{align*}
  where $L = N\max(A_{max},S_{max})^2$. Thus, (C1) in Lemma~\ref{lem:basis_flexible} is satisfied with $\eta = \frac{\delta^{(k)}}{2}$ and $\kappa = \frac{2L}{\delta^{(k)}}$. For (C2), we have 
  \begin{align*}
    |\norms{\Q_{\perp}^{(k)}(t+1)} - \norms{\Q_{\perp}^{(k)}(t)}| \le \norm{\Q_{\perp}^{(k)}(t+1)-\Q_{\perp}^{(k)}(t)}\le\norm{\Q(t+1)-\Q(t)}\le \sqrt{N}\max(A_{max},S_{max}).
  \end{align*}
  Thus, (C2) is satisfied with $D = \sqrt{N}\max(A_{max},S_{max})$. Thus, according to Eq.~\eqref{eq:upper_siva}, we have 
  \begin{align*}
    \ex{\norm{\Q_{\perp}^{(k)}}^r} &\le \left[\left(\frac{4L}{\delta^{(k)}}\right)^r + \left(\frac{8D^2 + 4D\delta^{(k)}}{\delta^{(k)}}\right)^r r!\right]\\
    &\lep{a} V_r^{(k)}r^{r+\frac{1}{2}}e^{1-r},
  \end{align*}
  where in (a) $V_r^{(k)} = \left(\frac{4L}{\delta^{(k)}} + \frac{8D^2 + 4D\delta^{(k)}}{\delta^{(k)}}\right)^r$ and $r!$ is upper bounded by Stirling's approximation.
\end{proof}

\section{Proof of Claim~\ref{clm:c1}}
\label{proof:clm1}
\begin{proof}
Let $\pi^{(k)} \triangleq \mathbb{P}\left(\inner{\ck}{\overline{S}}=b^{(k)}\right)$. By the result in~\cite{eryilmaz2012asymptotically} (i.e., Claim 1), we have $1-\pi^{(k)} \le \frac{\epsilon}{\gamma^{(k)}}$, in which $\gamma^{(k)}$ is a strictly positive constant that is independent of $\epsilon$. Then,
  \begin{align*}
    &\ex{\left(b^{(k)}-\inner{\ck}{\overline{\s} }\right)^{r'}}\nonumber\\
    =&(1-\pi^{(k)})\ex{\left(b^{(k)}-\inner{\ck}{\overline{\s} }\right)^{r'} \mid \inner{\ck}{\overline{\s} }\neq b^{(k)} }\nonumber\\
    \lep{a}&\beta_2\epsilon,
  \end{align*}
  where in (a) the constant $\beta_2$ exists since $S_n(t) \le S_{max}$ and  $1-\pi^{(k)} \le \frac{\epsilon}{\gamma^{(k)}}$.
\end{proof}

\section{Proof of Claim~\ref{clm:c2}}
\label{proof:clm2}
\begin{proof}
  Using the fact that the mean drift of $\inner{\ck}{\Q}$ is zero in steady-state, yields
  \begin{align}
    \ex{\inner{\ck}{\overline{\UU}} } &= \inner{\ck}{\ex{\overline{\s}} } - \inner{\ck}{\pmb{\lambda}}\nonumber\\
    &=\inner{\ck}{\ex{\overline{\s}} } - (b^{(k)} - \epsilon)\nonumber\\
    &\le \epsilon\label{eq:firstres}
  \end{align}
  For the second result, we use the same trick as in~\cite{eryilmaz2012asymptotically}. That is, we only consider strictly positive entries of $\ck$, as defined by $N_{++}^{(k)} \triangleq \{n \in \{1,2,\ldots,N\}: c_n^{(k)} >0 \}$. Then, we restrict the vectors on the $|N_{++}^{(k)}|$-dimensional real space: 
  \begin{align*}
    \widetilde{\mathbf{c}}^{(k)} \triangleq (c_n^{(k)})_{n\in N_{++}^{(k)}}, \quad \widetilde{\mathbf{Q}}^{(k)} \triangleq (Q_n^{(k)})_{n\in N_{++}^{(k)}}, \quad \widetilde{\mathbf{U}}^{(k)} \triangleq (U_n^{(k)})_{n\in N_{++}^{(k)}}, \quad \widetilde{\Q}_{\perp} \triangleq \widetilde{\Q} - \widetilde{\Q}_{\parallel}.
  \end{align*}
  Following the similar arguments in~\cite{eryilmaz2012asymptotically,hurtado2020logarithmic}, we obtain that 
  \begin{align*}
    &\ex{\inner{\ck}{\overline{\UU}(t)}\inner{\ck}{{\overline{\Q}^{(\epsilon)}(t+1)}}}\\
    \ep{a}&\mathbb{E}_{\overline{\Q}}\left[\inner{-\widetilde{\Q}_{\perp}^+}{\widetilde{\UU}}\right]\nonumber\\
    \lep{b}&\left(\mathbb{E}_{\overline{\Q}}\left[\norms{\widetilde{\UU}}_{r'}^{r'}\right]\right)^{\frac{1}{r'}}\left(\mathbb{E}_{\overline{\Q}}\left[\norms{\widetilde{\Q}_{\perp}^+}_{r}^{r}\right]\right)^{\frac{1}{r}}\nonumber\\
    \lep{c}&\left(\mathbb{E}_{\overline{\Q}}\left[\norms{\widetilde{\UU}}_{r'}^{r'}\right]\right)^{\frac{1}{r'}}\left(\mathbb{E}\left[\norms{\overline{\Q}_{\perp}}_{2}^{r}\right]\right)^{\frac{1}{r}}\nonumber\\
    \lep{d}& (M_r^{(k)})^{(1/r)}\left(\mathbb{E}_{\overline{\Q}}\left[\norms{\widetilde{\UU}}_{r'}^{r'}\right]\right)^{\frac{1}{r'}}\nonumber\\
    \lep{e} &(M_r^{(k)})^{(1/r)}\beta_3\epsilon^{\frac{1}{r'}}\\
    \lep{f}&\beta_3\beta_1\epsilon^{\frac{1}{r'}}r^{1+\frac{1}{2r}}e^{\frac{1}{r}-1}\\
    \lep{g}&2\beta_3\beta_1\epsilon\log{\frac{1}{\epsilon}}, \quad \forall \epsilon \le \epsilon^{\prime}_0
  \end{align*}
where (a) holds by Lemma~9 in~\cite{eryilmaz2012asymptotically}; (b) follows from H\"older inequality for random vectors, and $r, r'\in(1,\infty)$ satisfy $1/r + 1/r' = 1$; (c) holds by the fact that $\norms{\mathbf{x}}_{r_2} \le \norms{\mathbf{x}}_{r_1}$ if $r_1 < r_2$ and the similar argument for Eq.~(54) in~\cite{eryilmaz2012asymptotically}; (d) follows from the state-space collapse result in Eq.~\eqref{eq:ssc}; (e) follows from $\mathbb{E}_{\overline{\Q}}\left[\norms{\widetilde{\UU}}_{r'}^{r'}\right] \le \frac{S_{max}^{r'-1}}{c_{min}^{(k)}}\ex{\inner{\ck}{\overline{\UU} }}$ and the result in Eq.~\eqref{eq:firstres}. In particular, $\beta_3 = (\frac{S_{max}^{r'-1}}{c_{min}^{(k)}})^{1/r'}$, which is independent of $\epsilon$; (f) holds by the result in Eq.~\eqref{eq:ssc_bound} with $\beta_1 = (V_r^{(k)})^{1/r}$; (g) follows similar arguments in~\cite{hurtado2020logarithmic}. Specifically, we pick $r = \log{\frac{1}{\epsilon}}$ in (f), then RHS of (f) becomes $\beta_2\beta_1\epsilon \log{\frac{1}{\epsilon}} h(\epsilon)$, in which $h(\epsilon) \triangleq \epsilon^{-\frac{1}{\log{\frac{1}{\epsilon}}}}e^{\frac{1}{\log{\frac{1}{\epsilon}}}-1}\left(\log{\frac{1}{\epsilon}}\right)^{\frac{1}{2\log{\frac{1}{\epsilon}}}}$. Now, since $\lim_{\epsilon \downarrow 0} h(\epsilon) =  e \times \frac{1}{e} \times 1 = 1$, and hence there exists an $\epsilon^{\prime}_0$ such that for all $\epsilon \le \epsilon^{\prime}_0$, $h(\epsilon) \le 2$
\end{proof}
\end{document}